\documentclass{amsart}
\usepackage{amsmath,amsthm}
\usepackage{amsfonts,amssymb}
\usepackage{accents}
\usepackage{enumerate}
\usepackage{accents,color}
\usepackage{graphicx}

\newcommand{\lvt}{\left|\kern-1.35pt\left|\kern-1.3pt\left|}
\newcommand{\rvt}{\right|\kern-1.3pt\right|\kern-1.35pt\right|}

\newtheorem{thm}{Theorem}[section]
\newtheorem{cor}[thm]{Corollary}
\newtheorem{lem}[thm]{Lemma}
\newtheorem{prop}[thm]{Proposition}
\newtheorem{exam}[thm]{Example}

\newtheorem{defn}[thm]{Definition}

\theoremstyle{remark}

 \def\ve{{\varepsilon}}

 \def\d{\mathrm{d}}

 \def\sJ{{\mathsf J}}

 \def\sM{{\mathsf M}}

 \def\sS{{\mathsf S}}

 \def\a{{\alpha}}
 \def\b{{\beta}}
 \def\g{{\gamma}}

 \def\l{{\lambda}}

 \def\ve{{\varepsilon}}

 \def\xb{{\boldsymbol x}}

 \def\Fb{{\boldsymbol F}}

 \def\Pb{{\boldsymbol P}}
 \def\CB{{\mathcal B}}

 \def\CL{{\mathcal L}}
 
 \def\CP{{\mathcal P}}
 \def\CQ{{\mathcal Q}}

 \def\CV{{\mathcal V}}
 \def\CA{{\mathcal A}}

 \def\CC{{\mathbb C}}
 \def\FF{{\mathbb F}}
 
 \def\NN{{\mathbb N}}
 \def\PP{{\mathbb P}}
 
 \def\RR{{\mathbb R}}

      \def\Tr{{\mathsf T}}
      
\def\lla{\langle{\kern-2.5pt}\langle}      
\def\rra{\rangle{\kern-2.5pt}\rangle}

\newcommand{\wh}{\widehat}

\def\f{\frac}

\graphicspath{{./}}
\begin{document}
 
 \title[Continued Fractions and Orthogonal Polynomials]
 {Continued Fractions and Orthogonal Polynomials in Several Variables}

\author{Tomas~Sauer}
\address{Lehrstuhl f\"ur Mathematik mit Schwerpunkt Digitale Bildverarbeitung \& FORWISS, Universit\"t Passau, 
Innstr. 43, D-94053 Passau, Germany}
\author{Yuan~Xu}
\address{Department of Mathematics, University of Oregon, Eugene, 
OR 97403--1222, USA}
\email{yuan@uoregon.edu} 
\thanks{The second author was partially supported by Simons Foundation Grant \#849676}

\date{\today}  
 
\begin{abstract} 
We extend the close interplay between continued fractions, orthogonal polynomials, and Gaussian 
quadrature rules to several variables in a special but natural setting
which we characterize 
in terms of moment sequences. The crucial condition for the
characterization is the commutativity 
of the multiplication operators on finite polynomial subspaces modulo
an ideal. Moreover, starting from the 
orthogonal polynomials or the three-term recurrence, our method
constructs a sequence of moment sequences  
that provides an approximation of the maximal order for recovering the
moment sequence that defines the orthogonality. 
\end{abstract} 

\maketitle

\section{Introduction}
\setcounter{equation}{0}

It is well-known that continued fractions are closely related to orthogonal polynomials
and both, in turn, are related to the Gaussian quadrature rules. The interrelation 
between the three is elegant and revealing. We are interested in a possible extension 
of such relations in several variables. A moment of reflection,
however, shows several obvious 
obstacles to extending from one variable to the multi-variable setting. As a start,
it is not even clear how to define a continued fraction expansion
utilizing multivariate orthogonal polynomials.  

In the literature, there exist already several generalizations of
multivariate continued fractions, especially in terms of partitions
and lattices; see, for example, \cite{Karp,Schweiger}.  
Our approach will be fundamentally different and focus on rational
approximations of formal power  
series based on polynomials defined by a three-term recurrence formula. 
This is in the spirit of Gauss' original approach when developing Gaussian quadrature 
formulas from \cite{Gauss1816}. Therefore, it is natural that our
extension will be based on the three-term recurrences in $d$  
variables, which then consist of $d$ relations, one for each
coordinate multiplication. They are known for multivariate orthogonal
polynomials. The \emph{matrix}  
coefficients of the three-term relations will be embedded into an
appropriate definition  
of continued fractions, which, however, can only give a complete
extension of the univariate case when there is an appropriate
extension of the Gaussian quadrature rules to several variables. The latter relies on 
the existence of the maximal number of \emph{common} zeros of orthogonal 
polynomials of the same degree, which may or may not exist. Thus, the ideal 
generated by all such orthogonal polynomials of the same total degree must play 
a fundamental role, which translates to special structures of the matrix coefficients 
and to the moment sequences that define the orthogonality. This approach is 
fundamentally different from the ones found in \cite{Karp,Schweiger} and therefore 
yields yet another generalization of continued fractions to several variables, a 
generalization that is natural in our context. 

Our main result shows that there is a natural and complete extension 
to several variables whose existence can be characterized in
terms of the moment sequence.

In one direction, we start with  
a moment functional $\CL$, defined via a moment sequence $\mu$ in
$\RR^d$, and assume that the moment  
functional $\CL$ is \emph{definite} so that orthogonal polynomials defined
via $\CL$ exist. These polynomials then 
satisfy three-term relations which are characteristic properties of
orthogonality by Favard's theorem. The  
three-term relations define the Jacobi matrices, which are
representations of multiplication operators by  
coordinates. Our main assumption is that the truncated Jacobi matrices
commute, which leads to the existence  
of a maximal number of common zeros of orthogonal polynomials and then
to Gaussian cubature rules. If $\CL$ is positive
definite, this is how the Gaussian cubature rules are characterized in
\cite{DX}. In the above narrative. however, we do not require anymore 
that $\CL$ is positive definite. 
This is in accordance with the well-known and classical theory of univariate
continued fractions and the fast convergence of the associated
convergents, see \cite{S_CF}, since the existence of a so-called 
\emph{associated continued fraction} is equivalent to
the definiteness. Moreover, the positivity of $\CL$ can even 
be expressed in terms of the signs of the recurrence coefficients in 
the univariate case. In several variables, it is well-known that the
issue of positivity
is significantly more complex, which is essentially due to the fact that
positive polynomials are no longer sums of squares (cf. \cite{Schmuedgen}).
In the framework that we present here, many results have to be
reconsidered and modified. For example,  
without positivity, the common zeros may not be real or simple anymore. As a result, 
it is necessary to recast many of the results on the structure of OPs in several 
variables for a functional $\CL$ that is only definite but not
necessarily positive definite, which we  
shall carry out in some detail and provide proof whenever
necessary. In return, this framework enables us to give a definition
of continued fractions that can be regarded as an  
extension from one to several variables.

In the other direction, we start with a polynomial sequence that satisfies three-term 
relations and assume that their truncated multiplication operators commute. 
By Favard's theorem, these polynomials  
are orthogonal with respect to a linear functional $\CL$ which can be regarded as a 
moment functional defined by $\mu_\a = \CL(x^\a)$. Using continued
fractions, we
can then define a sequence of moment sequences $\mu^n$ and show that
$\mu^n \to \mu$ with an optimal rate in the sense that a maximal
number of terms in $\mu^n$ coincides with those in $\mu$;
in other words, we can efficiently recover the moment sequence that
determines the orthogonality.  
The main requirement for this setup lies in the commutativity of the
truncated Jacobi matrices, which can in turn be reduced to a
restriction on the moments. 

The paper is organized as follows. In the next section, we recall the
classical results of one variable,
which serves as a roadmap for our extension to several variables. In
the third section,  
we provide a detailed account of orthogonal polynomials in several variables and 
Gaussian cubature rules, where the results are presented and proved when necessary,  
without assuming the positive definiteness. In the fourth section, the \emph{multivariate} 
continued fraction expansion is defined and used to recover the moment sequence and 
we provide a characterization of the moments for which such an extension is possible.
 
\section{Preliminary}
\setcounter{equation}{0}
We recall results in one variable in this section. To that end, let $\CL$ be a
linear moment functional, defined at least on the polynomials, and
let $\mu_n = \CL(x^n)$ , $n=0,1,2,\ldots$, denote the moments of
$\CL$. In particular, if $w$ is a (not necessarily nonnegative) weight
function on $\RR$, then 
$$
  \CL (f) = \int_{\RR} f(x) w(x) \d x
$$
is such a moment functional. We assume that all $\mu_n$ are finite and
call $\CL$, or $w$,  
{\it definite} if the Hankel matrices  
$$
M_n : =\left[\mu_{k+j}: \begin{matrix} 0 \le k \le n \\ 0\le  j\le
    n \end{matrix}\right] 
$$ 
are nonsingular for all $n =0,1,\ldots$, and {\it positive definite} if, in addition, the matrices 
are positive definite. If $w \neq 0$ is nonnegative, then the Hankel
matrices are positive definite  
for all $n$, so that $\CL(f)$ is positive definite. If $w$ changes
sign on its domain, however,  
the determinant of $M_n$ could be zero for some $n$, in which case
orthogonal polynomials could  
degenerate in their degrees; see, for example, \cite{Mo, SG}. Our
definition of definiteness avoids these degenerate cases.  

A real polynomial $p_n$ of degree $n$ is called orthogonal with
respect to $\CL$ if  
$\CL(p_n q) = 0$ for all $q$ of degree at most $n-1$. It is known that
if $\CL$ is definite, then  
the orthogonal polynomials $\{p_n\}_{n\ge 0}$ exist and they satisfy a
three-term relation 
\begin{equation}
  \label{eq:recurrence-1d}
  p_{n+1} (x) = \left( a_n x + b_n \right) p_n (x) + c_n \, p_{n-1}(x), \qquad a_n c_n \ne 0, \qquad n \ge 0.
\end{equation}
Moreover, this relation characterizes the orthogonality according to Favard's theorem. For $\CL$ definite, 
the polynomial $p_n$ has $n$ zeros counting multiplicity. If $\CL$ is
positive definite, the zeros are all real and simple. If $\CL$ is  
definite but not positive definite, the zeros can be complex and
appear in conjugate pairs; that is, if $z$ is a zero of $p_n$ then so
is $\bar z$.  

When $w$ is positive definite, the zeros of $p_n$, $\{x_k: 1 \le k \le n\}$, are nodes of 
the Gaussian quadrature rule, 
\begin{equation} \label{eq:Gauss-quad}
   \int_{\RR} f(x) w(x) \d x = \sum_{k=1}^n  \l_k f(x_k), \qquad \deg f \le 2n-1,
\end{equation}
where $\l_k >0$. This remains true if $\CL$ is definite but not positive definite, and $p_n$ and $p_{n+1}$ 
do not have common zeros, for which $\l_k \ne 0$ for $1 \le k \le n$. The quadrature rule
\eqref{eq:Gauss-quad} has been studied extensively. For signed weight function, it holds for all even 
$n$ if $w$ is a weight function on $[-1,1]$ and $\det \left[\mu_{2(k+j)-1}\right]_{k,j =1}^n \ne 0$, for 
example; see \cite[Section 3.1]{Gau}. Another interesting case is the oscillatory weight functions. Let 
$p_n(w)$ denote the orthogonal polynomial of degree $n$ with respect to $w$. The oscillatory weight 
function $\wh w$ is defined by 
$$
  \wh w(t)  = p_m (w; t) w(t), \qquad t \in \RR. 
$$
In this case, the moment $\mu_k (\wh w) = 0$ if $k < m$ and $\mu_m(\wh w) >0$, so that the $m+1$ Hankel
matrix satisfies $\det M_{m+1} > 0$. In particular, the polynomial $p_n(\wh w)$, $n = m+1$, is well defined. Moreover,
it has all real and simple zeros in some cases; for example, when $w(t) = (1-t^2)^{\l-\f12}$ and $0 < \l \le 2$ (see,
for example, \cite{Gau, No}).

We are mostly interested in the case when $w$ is definite and $p_n(w)$ has distinct real simple zeros, which
holds whenever $\CL$ is positive definite. Below is an example when $\CL$ is definite but not positive definite. 

\begin{exam}
For $\l > -\f12$ and $\mu > 0$, we define the weight function 
$$
  w_{\l,\mu} (t) = t^{2\mu+1} (1-t) (1-t^2)^{\l-\f12}, \qquad -1 < t < 1.
$$ 
Then $w_{\l,\mu}$ changes sign on $(-1,1)$. The orthogonal polynomial $p_n(w_{\l,\mu})$ has $n$ simple,
real zeros in $[-1,1]$ and they are nodes of the Gaussian quadrature of degree $2n-1$.
\end{exam}

Indeed, let $C_n^{(\l,\mu)}$ be the generalized Gegenbauer polynomials that are orthogonal with respect to 
the weight function $|t|^{2\mu} (1-t^2)^{\l-\f12}$ on $[-1,1]$. These polynomials can be given explicitly in
terms of the Jacobi polynomials $P_n^{(\a,\b)}$ (cf. \cite[Section 1.5.2]{DX}). In particular, 
up to a multiple constant, 
\begin{align*}
C_{2n}^{(\lambda ,\mu )}(t) = 
   b_n P_{n}^{(\lambda -\f12,\mu-\f12)}(2t^{2}-1).
\end{align*}
 Then it can be easily verified (cf. \cite{ABC}) that 
$$
  p_{2n}(w_{\l,\mu}; t) = C_{2n}^{(\l,\mu+\f12)}(t) \quad \hbox{and}\quad 
   p_{2n+1}(w_{\l,\mu}; t) = (1+t) C_{2n}^{(\l+1,\mu+\f12)}(t) 
$$
for $n =0,1, 2,\ldots$. Moreover, it is easy to see that, if $\mu$ is an integer, then
$$
  \int_{-1}^1 \left| p_{2n}(w_{\l,\mu}; t)\right |^2 w_{\l,\mu}(t) \d t = - \int_{-1}^1 \left| C_{2n}^{(\l,\mu)} \right| |t|^{2\mu+2}
    (1-t^2)^{\l-\f12} \d t < 0. 
$$
The properties of the zeros of $p_n(w_{\l,\mu})$ follow from those of the zeros of the Jacobi polynomials. In particular,
it follows that $w_{\l,\mu}$ admits Gaussian quadrature rules of degree $2n-1$ for all $n$. 

The recurrence relation relates orthogonal polynomials closely to continued
fractions, a fact used by Gauss to obtain the quadrature rules in
\cite{Gauss1816}. Indeed, defining the (infinite) continued fraction
as
$$
r(x) = \frac{b_1 |}{| a_1 (x)} + \frac{b_2 |}{| a_2 (x)} + \frac{b_3
  |}{| a_3 (x)}\cdots 
= \frac{b_1}{a_1 (x) + \dfrac{b_2}{a_2 (x) + \dfrac{b_3}{\ddots}}}.
$$
Its \emph{convergents} of order $n$, defined by
$$
r_n (x) = \frac{b_1 |}{| a_1 (x)} + \frac{b_2 |}{| a_2 (x)} + \dots + \frac{b_n
  |}{| a_n (x)} = \frac{p_n (x)}{q_n (x)},
$$
are rational functions, where the numerator $p_n$ and denominator $q_n$ are 
polynomials and they satisfy a recurrence relation
\begin{align*}
  \label{eq:CFrecurrence}
  p_n (x) &= a_n (x) p_{n-1} (x) + b_n p_{n-2}, \qquad p_{-1} = 1, \,
            p_0 = 0, \\
  q_n (x) &= a_n (x) p_{n-1} (x) + b_n p_{n-2}, \qquad q_{-1} = 0, \,
            q_0 = 1.
\end{align*}
If the polynomials $a_n$ are \emph{affine polynomials}, i.e., $a_n (x) =
\alpha_n x + \beta_n$, then the above recurrence is the classical
\emph{three-term recurrence relation} for orthogonal polynomials. In
fact, the polynomials
$q_n$ \emph{are} orthogonal polynomials provided that $\alpha_n \, c_n
< 0$.

A classical result in continued fractions says that a formal Laurent
series
$$
\mu (z) = \sum_{n=1}^\infty \mu_{n-1} z^{-n}
$$
has an \emph{associated} continued fraction expansion, i.e., $\mu (z) - r_n
(z) = O \left( z^{2n} \right)$, if and only if
$$
\det \left[ \mu_{j+k} :
  \begin{array}{c}
    j=0,\dots,n \\ k=0,\dots,n
  \end{array}
\right] \neq 0, \qquad n \in \NN_0.
$$
If $\mu$ is the moment sequence associated with a definite linear
functional, then the denominator polynomials $q_n$ are the respective
orthogonal polynomials and the maximal exactness of Gaussian quadrature
is eventually a consequence of the fact that the Laurent series for
the quadrature moments is $p_n/q_n$. This is the central idea in
Gauss' original paper \cite{Gauss1816}.

Continued fraction expansions are closely related to \emph{Prony's problem} 
which consists of recovering a multi-exponential function
$$
f(x) = \sum_{k=1}^m f_k \, \zeta_k^x, \qquad \zeta_k \in \CC \setminus \{ 0 \},
$$
from integer samples $f(n)$, $n \in \NN$. In fact, interpreting these
samples of $f$ as the 
moment sequence $\mu_n = f(n)$, the central connection is that the infinite
\emph{Hankel operator} $
\begin{bmatrix}
  \mu_{j+k} : j,k \in \NN_0 
\end{bmatrix}$ has finite rank in this case and, by Kronecker's
Theorem, the associated Laurent series $\mu (z)$ is a rational
function, hence has a \emph{finite} continued fraction
expansion. Moreover, the denominator of the rational function is the
so-called \emph{Prony polynomial} whose zeros are exactly 
$\zeta_1,\dots,\zeta_m$. 
The monograph \cite{S_CF} provides more details on the connection
between continued fractions and Prony's problem, and how the continued
fraction expansion and the construction of the associated recurrence
coefficients can be used to solve Prony's problem even numerically.

\section{Orthogonal polynomials and cubature rules}
\setcounter{equation}{0}

We now consider OPs in several variables with respect to a linear
functional $\CL$ or, equivalently, a moment sequence $\mu$.
As mentioned in the introduction, we shall need results on the structures of these
polynomials for $\CL$ being definite, but not necessarily positive definite. This 
requires some modifications of the existing theory that has been
established for the positive definite case;
most of these modifications are straightforward but by no means all of
them. For example, the proof for the existence of a maximal number of
common zeros requires additional algebraic tools. For the record and
the reader's convenience, we 
will point out the modifications and provide proof whenever necessary. 

Let $\CL$ be a linear functional. For $\a \in \NN_0^d$, let 
$\mu_\a = \CL(x^\a)$ denote the $\a$th moment of $\CL$. We assume for
simplicity that $\mu_0
= \CL(1) \ne 0$ and that all moments $\mu_\a$ are real-valued.
Let $\Pi_n^d$ denote the space of polynomials of total degree at most $n$ in $d$ 
variables and let $\CP_n^d$ denote the space of homogeneous
polynomials of degree $n$. Then 
$$
r_n := \dim \Pi_n^d = \binom{n+d}{d} \quad \hbox{and} \quad r_n^0 := \dim \CP_n^d = \binom{n+d-1}{d-1}.
$$
For $P \in \Pi_n^d$, we write $\wh P \in \RR^{r_n}$ for the
coefficient vector so that $P(x) = \wh P^\Tr \xb^n = \sum_\alpha \wh
P_\alpha x^\alpha$, 
where $\xb^n := \left[ x^\alpha : |\alpha| \le n \right]$. We arrange
$\wh P$ in a graded fashion as
$$
\wh P = \left[
  \begin{array}{c}
    \wh P_0 \\ \vdots \\ \wh P_n
  \end{array}
\right], \qquad \wh P_j := \left[ \wh P_\alpha : |\alpha| = j \right],
\quad j=0,\dots,n,
$$
and decompose $\xb^n$ accordingly. The moment sequence $\mu$ defines a
linear functional $\CL$ on $\Pi$ by linearity as
$$
\CL (P) = \sum_{|\alpha| \le n} p_\alpha \mu_\alpha, \qquad
P(x) = \sum_{|\alpha| \le n} p_\alpha x^\alpha \in \Pi_n^d.
$$
A polynomial $P \in \Pi_n^d$ is called orthogonal with respect to the linear functional $\CL$ if 
$$
   \CL (P Q) = 0, \qquad \forall Q \in \Pi_{n-1}^d. 
$$
In the usual setting of OPs in several variables, we assume that $\CL$ is (square) positive definite, which requires 
$\CL(p^2) > 0$ for all non-zero polynomials. This ensures the existence of an orthonormal basis of polynomials. 
In this section, we state the structure of OPs in the general setting, which are known in the positive definite case, 
but require some modification if $\CL$ is definite but not necessarily positive definite. 

\subsection{Structure of orthogonal polynomials}
For $n=0,1,2,\ldots$, let 
$$
   M_n = \left[\mu_{\a+\b}: \begin{matrix} |\a|\le n \\ |\b|\le n \end{matrix} \right]
$$
be the moment matrix of size $r_n \times r_n$. Clearly,
$$
\CL (PQ) = \wh P^\Tr M_n \wh Q, \qquad P,Q \in \Pi_n^d.
$$
We assume that the matrices $M_n$, $n=0,1,2,\ldots$, are 
non-singular, which ensures the existence of orthogonal polynomials. Indeed, the usual construction of OPs 
from the moments (cf. \cite[Section 3.3.2]{DX}) applies, which gives a monic OP of the form
$$
   P_\a(x) = x^\a + q_\a, \qquad q_\a \in \Pi_{n-1}^d, 
$$ 
for each $\a \in \NN_0^d$. Let $\CV_n^d$ be the space of OPs of degree $n$. Then $\dim \CV_n^d = r_n^0$. 
Let $\wh \PP_n$ be the column vector 
$$
\wh \PP_n(x) = \left[ P_\a: |\a| = n \right] = \wh \Pb_n^\Tr \xb^n,
\qquad \wh \Pb_n := \left[ \hat P_\a: |\a| = n \right] \in \RR^{r_n
  \times r_n^0}.
$$
Then $\wh \PP_n$ consists of a basis of $\CV_n^d$. Partitioning $M_n$
as
$$
M_n = \left[
  \begin{array}{cc}
    M_{n-1} & M_{n-1,n,} \\
    M_{n-1,n}^\Tr & M_{n,n}
  \end{array}
\right],
$$
with
$$
M_{n-1,n} = \left[ \mu_{\alpha+\beta} :
  \begin{array}{c}
    |\alpha| \le {n-1} \\ |\beta| = n
  \end{array}
\right] \quad \hbox{and}\quad M_{n,n} = \left[ \mu_{\alpha+\beta} :
  \begin{array}{c}
    |\alpha| = n \\ |\beta| = n
  \end{array}
\right],
$$
it can easily be verified that
$$
\wh \Pb_n = \left[
  \begin{array}{c}
    -M_{n-1}^{-1} M_{n-1,n} \\ I_{r_n^0 \times r_n^0}
  \end{array}
\right],
$$
and that the matrix
$$
\CL(\wh \PP_n \wh \PP_n^\Tr) = \wh \Pb_n^\Tr M_n \wh \Pb_n =
M_{n,n} - M_{n-1,n}^\Tr M_{n-1}^{-\Tr} M_{n-1,n}
$$
is the symmetric and non-singular Schur complement of $M_{n,n}$ in
$M_n$. This matrix satisfies a decomposition 
$$
\CL(\wh \PP_n \wh \PP_n^\Tr) = Q^\Tr \Lambda Q, \qquad \Lambda = |\Lambda|^{\f12} S_n |\Lambda|^{\f12},
$$
where $Q$ is an orthogonal matrix and $\Lambda = \mathrm{diag} \{\l_1,\ldots, \l_{r_n^0})$ is a non-singular 
diagonal matrix, which is further decomposed by using 
$|\Lambda|^{\f12} = \mathrm{diag}\{|\l_1|^{\f12},\ldots,
|\l_{r_n^d}|^{\f12})$ and a \emph{signature matrix} $S_n \in
\RR^{r_n^0 \times r_n^0}$, the latter being a diagonal matrix with the
diagonal elements being either $1$ or $-1$. We now define
$$
    \PP_n =  |\Lambda|^{- \f12} Q \wh \PP_n, \qquad n =0,1,2,\ldots. 
$$
Then $\PP_n$ consists of a basis of orthogonal polynomials of degree $n$, so that 
$$
  \CL (\PP_n \PP_m^\Tr) = \delta_{n,m} S_n, \qquad n \ne m,
$$
which we call a \emph{sign-orthonormal} basis. If $\CL$ is positive
definite, then $S_n$ is the identity matrix,  
and the sign-orthonormal basis becomes an orthonormal basis. 

\subsection{Recurrences}
\label{sec:recurrences}

Orthogonal polynomials are described by a three-term recurrence
relation as follows.

\begin{prop}\label{prop:SignOrthognal}
Let $\PP_n$ be a sign-orthonormal basis of $\CV_n^d$, $n=0,1,2,\ldots$
and $\PP_{-1} :=0$. Then there exist unique matrices $A_{n, i}:
r_n^0\times r^0_{n+1}$ and $B_{n,i}: r_n^0\times r_n^0$, such that
\begin{equation} \label{eq:3term}
   x_i S_n \PP_n(x) = A_{n,i} \PP_{n+1}(x) + B_{n,i} \PP_n(x) +
   A_{n-1,i}^\Tr \PP_{n-1}(x),  \qquad 1 \le i \le d, 
\end{equation}
where the matrix $B_n$ is symmetric and 
\begin{equation} \label{eq:rank-cond}
   \mathrm{rank}\, A_{n,i}  = r_n^0 \quad \hbox{and}\quad 
   \mathrm{rank}\, (A_{n,1}^\Tr, \ldots, A_{n,d}^\Tr)^\Tr = r_{n+1}^0,  
\end{equation}
\end{prop}

Indeed, since the
vectors $\PP_n$, $n \ge 0$, are a basis of $\Pi^d$, there must exist
matrices $A_{n,i}, B_{n,i}, C_{n,i}$ of proper size such that
$$
x_i S_n \PP_n(x) = A_{n,i} \PP_{n+1}(x) + B_{n,i} \PP_n(x) +
   C_{n,i} \PP_{n-1}(x),
$$
and using the
orthogonality, it is easy to see that  
\begin{equation}
  \label{eq:AicLFormula}
  A_{n,i} S_{n+1} = S_n \CL (x_i \PP_n \PP_{n+1}^\Tr), \qquad B_{n,i}
  S_n = S_n \CL (x_i \PP_n \PP_n^\Tr),   
\end{equation}
which shows that $B_n$ is symmetric since the signature matrix $S_n$
satisfies $S_n^\Tr = S_n = S_n^{-1}$; moreover, using
\eqref{eq:AicLFormula} with $n$ replaced by $n-1$  
in the second equation, we obtain
$$
     C_{n,i} S_{n-1} = S_n \CL (x_i \PP_n \PP_{n-1}^\Tr) = S_n (S_n A_{n-1,i}^\Tr S_{n-1}) = A_{n-1,i}^\Tr S_{n-1},
$$
which implies that $C_{n,i} = A_{n-1,i}^\Tr$. 

The rank conditions can be established as in the positive definite case: for example, since $\PP_n$ is 
sign-orthonormal, its leading coefficient matrix $\wh \Pb_n$, defined by 
$$
   \PP_n(x) = \wh \Pb_n^\Tr \xb^n_0 + \ldots, \qquad \xb^n_0 := \left[x^\a:
     |\a| =n \right], 
$$
is a non-singular matrix of size $r_n^d$. Comparing the leading coefficients of \eqref{eq:3term} gives
$$
  S_n \wh \Pb_n L_{n,i} = A_{n,i} \wh \Pb_{n+1}, 
$$
where $L_{n,i}$ is the shift matrix of dimension $r_n^0 \times
r_{n+1}^0$, defined by $L_{n,i} \xb^{n+1}_0 = x_i \xb^n_0$; this matrix
has full rank, thus $A_{n,i}$ has rank $r_n^0$. 

The coefficient matrices of the three-term relation for
sign-orthonormal polynomials satisfy a set of commutativity  
relations: for $1 \le i, j \le d$ and $k \ge 0$, 
\begin{align} \label{eq:comm}
 \begin{split}
   A_{n,i} S_{n+1} A_{n+1,j} &= A_{n,j} S_{n+1} A_{n+1,i},   \\  
   A_{n,i} S_{n+1} B_{n+1,j} + B_{n,i} S_n A_{n,j} &= B_{n,j} S_n A_{n,i}
   + A_{n,j} S_n B_{n+1,i},  \\
  A_{n-1,i}^\Tr S_{n-1} A_{n-1,j}  + B_{n,i} S_n B_{n,j} + &  A_{n,i}
  S_{n+1} A_{n,j}^\Tr    \\
  = A_{n-1,j}^\Tr S_{n-1} A_{n-1,i}  &+  B_{n,j} S_n B_{n,i} + A_{n,j}
  S_{n+1} A_{n,i}^\Tr,   
\end{split}
\end{align}
where $A_{-1,i} : =0$. This follows from computing $\CL(x_i x_j  \PP_n \PP_{n \pm 1}^\Tr)$ 
via the three-term relation in two different ways. For example, 
\begin{align*}
  S_n \CL(x_i x_j \PP_n \PP_{n+1}^\Tr) S_{n+1}^\Tr
  & =  \CL(x_i S_n \PP_n (x_j S_{n+1} \PP_{n+1})^\Tr \\  
  &  =  A_{n,i} \CL(\PP_{n+1} \PP_{n+1}^\Tr) A_{n+1,j}) = A_{n,i} S_{n+1} A_{n+1,j},
\end{align*}
and the left-hand side is unchanged if the order of $i$ and $j$ is
exchanged, which shows the first identity; the others are proved in
exactly the same way.

In fact, for \emph{any} set of basis functions $\FF_n \in \left(
  \Pi_n^d \right)^{r_n^0}$ and the multiplication rule
\begin{equation}
  \label{eq:MultiplArbit}
  x_i \FF_n (x) = A_{n,i} \FF_{n+1} (x) + B_{n,i} \FF_n (x) + C_{n,i}
  \FF_{n-1}, \qquad n \in \NN_0, \, 1 \le i \le d,
\end{equation}
the consistency conditions $x_i x_j \FF_{n-1} = x_j x_i \FF_{n-1}$
implies the identities
\begin{align}
  \label{eq:CommMultGen}
  \begin{split}
      A_{n-1,i} A_{n,j} & =  A_{n-1,j} A_{n,i}, \\
      A_{n-1,i} B_{n,j} + B_{n-1,i} A_{n-1,j}
      & = A_{n-1,j} B_{n,i} + B_{n-1,j} A_{n-1,i}, \\
      A_{n-1,i} C_{n,j} + B_{n-1,i} B_{n-1,j} +& C_{n-1,i} A_{n-2,j} \\
      = A_{n-1,j} C_{n,i} & + B_{n-1,j} B_{n-1,i} + C_{n-1,j} A_{n-2,i} \\
      B_{n-1,i} C_{n-1,j} + C_{n-1,i} B_{n-2,j},
      & = B_{n-1,j} C_{n-1,i} + C_{n-1,j} B_{n-2,i},\\
      C_{n-1,i} C_{n-2,j} & = C_{n-1,j} C_{n-2,i},  
  \end{split}
\end{align}
which reduces to \eqref{eq:comm} for $\FF_n = \PP_n$. Note that the
conditions \eqref{eq:CommMultGen} and \eqref{eq:comm}, respectively,
become trivial for $d=1$ and thus 
are a purely multivariate phenomenon, i.e., are relevant for $d \ge 2$
only.

We can use the matrices occurring in the three-term recurrence
\eqref{eq:3term} to define a family of tridiagonal infinite 
matrices $\sJ_i$, $1 \le i \le d$, as follows:
\begin{equation} \label{Jmatrix}
\sJ_i := \left[ \begin{matrix} B_{0,i}&A_{0,i}&&\bigcirc\cr
A_{0,i}^\Tr &B_{1,i}&A_{1,i}&&\cr
&A_{1,i}^\Tr &B_{2,i}&\ddots\cr
\bigcirc&&\ddots&\ddots \end{matrix}
\right],\qquad 1\le i\le d.
\end{equation}
The relations \eqref{eq:comm} are then equivalent to the formal
commutativity of $\sJ_i$, that is, $\sJ_i \sS \sJ_j = \sJ_j \sS \sJ_i$,  
$1 \le i, j \le d$, where $\sS = \mathrm{diag} \{S_0,
S_1,S_2,\ldots\}$ is the infinite block diagonal matrix formed by the
sign matrices of the functional $\CL$.

Since $A_{n,i}: r_n^0 \times r_{n+1}^0$ has full rank, there exist
matrices $D_{n,i}: r_n^0 \times r_{n+1}^0$, $i=1,\dots,d$, such that 
$$
   \sum_{i=1}^d D_{n,i}^\Tr A_{n,i} = I.
$$
The matrices $D_{n, i}$ are not unique and we can take, for example,
$(D_{n,1}, \ldots D_{n,d})$ as the generalized 
inverses of $(A_{n,1}, \ldots, A_{n,d})$. The three-term relations
\eqref{eq:3term} lead to a recursive definition of $\PP_n$, given by
\begin{align} \label{eq:recursiveP}
  \PP_{n+1}(x) = \sum_{i=1}^d D_{n,i}^\Tr x_i S_n \PP_n(x) - E_n \PP_n(x) - F_n \PP_{n-1}(x),
\end{align}
where the matrices $E_n$ and $F_n$ are given by 
$$
E_n = \sum_{i=1}^d D_{n,i}^\Tr B_{n,i} \quad \hbox{and} \quad F_n =
\sum_{i=1}^d D_{n,i}^\Tr A_{n-1,i}^\Tr, 
$$
respectively.
The relation \eqref{eq:recursiveP} can be used for a recursive definition of the sequence of $\PP_n$ 
if $A_{n,i}$, $B_{n, i}$ are given. In one variable, the sequence of polynomials so defined automatically 
satisfies the three-term relation \eqref{eq:3term}. In several
variables, however, we have to require the matrices 
$A_{n,i}$ and $B_{n,i}$ to satisfy some necessary conditions. Indeed,
a slightly modified version of the proof given in
\cite[Theorem 3.5.1]{DX} for the positive definite setting shows that
if $A_{n, i}$ and $B_{n, i}$ satisfy the
commuting conditions \eqref{eq:comm} and the rank condition \eqref{eq:rank-cond}, then a sequence
of polynomials $\{\PP_n\}_{n\ge 0}$ defined recursively by \eqref{eq:recursiveP} satisfies the 
three-term relations \eqref{eq:3term}.  

The three-term relations characterize the orthogonality by Favard's
theorem. We call the linear functional $\CL$ \emph{quasi-determinate}
if there is a basis $\CB$ of $\Pi^d$ such that, for any $P, Q \in
\CB$,
$$
\CL (PQ) = 0 \quad \hbox{if $P \ne Q$ and $\CL (P^2) \ne 0$}.
$$

\begin{thm}\label{thm:quasidetFun}
  Let $\{\PP_n\}_{n \ge 0} =\{P_\a \in \Pi_{n}^d : |\a| = n \in \NN_0\}$ be an arbitrary polynomial sequence.  
  Then the following statements are equivalent. 
  \begin{enumerate}[(i)] 
  \item There exists a linear functional $\CL$ which defines a
    quasi-definite linear functional on $\Pi^d$ and which makes
    $\{P_\a^n: \a \in \NN_0^d\}$ an orthogonal basis of
    polynomials.
  \item For $n \ge 0$, $1\le i\le d$, there exist matrices $A_{n,i}$
    and $B_{n, i}$ such that
    \begin{enumerate}[(a)] 
    \item the polynomials $\PP_n$ satisfy the three-term relation
      \eqref{eq:3term}; 
    \item the matrices in the relation satisfy the rank conditions
      \eqref{eq:rank-cond}. 
    \end{enumerate}
  \end{enumerate}
\end{thm}

For the OPs in the positive definite case, the above theorem holds
with $S_n = I$, $n \ge 0$. The proof of the theorem 
follows with obvious modification of that in the positive definite
case; see \cite[Theorem 3.3.7]{DX}.

As a consequence of the three-term relation, we can derive a Christoffel-Darboux formula. Let 
\begin{equation} \label{eq:reprod-kernel}
   K_n(x,y) = \sum_{k=0}^n \PP_k(x)^\Tr \sS_k \overline{\PP_k(y)}, \qquad x, y \in \CC^d.
\end{equation}
Then $K_n$ is the reproducing kernel of the space $\Pi_n^d$ since, for $0 \le j \le n$ and $x,y \in \RR^d$, 
$$
   \CL (K_n(x,\cdot) \PP_j^\Tr) = \sum_{k=0}^n \PP_k(x)^\Tr S_k \CL(\PP_k \PP_j^\Tr) = \PP_j(x)^\Tr.
$$

\begin{thm}
For $n=0,1,2,\ldots$, the following Christoffel-Darboux formula holds  
\begin{equation} \label{eq:Knxy}
   K_n(x,y) = \frac{\left[A_{n,j} \PP_{n+1}(x)\right]^\Tr \overline{\PP_n(y)} - 
          \PP_n(x)^\Tr \left[A_{n,j} \overline{\PP_{n+1}(y)}\right]}{x_j- \overline{y}_j}, \quad x, y \in \CC^d,
\end{equation}
for $1 \le j \le d$. In particular, taking the limit shows that 
\begin{equation} \label{eq:Knxx}
 K_n(x,x) = \left[A_{n,i} \PP_{n+1}(x)\right]^\Tr \partial_j\PP_n(x) -  \PP_n(x)^\Tr \left[A_{n,i} \partial_j \PP_{n+1}(x)\right] .
\end{equation}
\end{thm} 

\begin{proof}
By the three-term relation \eqref{eq:3term}, 
\begin{align*}
 & x_j S_k \PP_k(x)^\Tr \overline {\PP_k(y)} - \overline{y}_j S_k \PP_k(x)^\Tr \overline {\PP_k(y)} \\
&\quad  =  \left([A_{k,i} \PP_{k+1}(x)]^\Tr \overline{\PP_k(y)} - [A_{k,i} \overline{\PP_{k+1}(y)}]^\Tr \PP_k(x) \right) \\
   & \quad -  \left([A_{k-1,i} \PP_{k}(x)]^\Tr \overline{\PP_{k-1}(y)} - [A_{k-1,i} \overline{\PP_{k}(y)}]^\Tr \PP_{k-1}(x) \right),
\end{align*}
and summing this identity over $k$ gives \eqref{eq:Knxy}. Since
$\PP_n$ has real coefficients, we see that 
$\overline \PP_n(x) = \PP_n(\bar x)$. Hence, we can write 
\begin{align*}
& \left[A_{n,j} \PP_{n+1}(x)\right]^\Tr \overline{\PP_n(y)} - \PP_n(x)^\Tr \left[A_{n,j} \overline{\PP_{n+1}(y)}\right] \\
& = \left[A_{n,j} \PP_{n+1}(x)\right]^\Tr (\PP_n(\bar y) - \PP_n(x)) + \PP_n(x)^\Tr A_{n,j} (\PP_{n+1}(x) - \PP_{n+1}(\bar y)).
\end{align*}
Dividing by $x_j - \overline{y}_j$ and taking the limit $\overline{y}_j \mapsto x$ proves \eqref{eq:Knxx}. 
\end{proof} 

\subsection{Jacobi matrices and common zeros}
\label{sec:commzeros}

We now recall the \emph{truncated Jacobi matrices} $\sJ_{n,i}$, defined
for each $n \in \NN_0$ by  
\begin{equation}
  \label{eq:truncJacobi}
  \sJ_{n,i} :  =  \left[ \begin{matrix} B_{0,i}&A_{0,i}&&&\bigcirc\cr
      A_{0,i}^\Tr &B_{1,i}&A_{1,i}&&\cr  &\ddots&\ddots&\ddots&\cr
      &&A_{n-3,i}^\Tr &B_{n-2,i}&A_{n-2,i}\cr
      \bigcirc&&&A_{n-2,i}^\Tr &B_{n-1,i}  \end{matrix} \right] \in
  \RR^{r_{n-1} \times r_{n-1}},
  \qquad 1\le i\le d.  
\end{equation}
An
element $\lambda \in \CC^d$ is called a generalized {\it joint eigenvalue} of $\sJ_{n,1}, \ldots, \sJ_{n,d}$, 
if there is $\xi \ne 0$, $\xi \in \CC^d$, such that 
$$
\sJ_{n,i}\xi = \lambda_i \sS_n \xi, \qquad i= 1, \ldots, d, \qquad
\sS := \left[
  \begin{array}{ccc}
    S_0 && \\
        & \ddots & \\
    & & S_n
  \end{array}
\right];
$$
the vector $\xi$ is called a joint eigenvector of the generalized eigenvalue $\l$. 

\begin{thm}\label{thm:JointEigen}
  A point $z \in \CC^d$ is a zero of $\PP_n$ if and only if it is a
  generalized joint eigenvalue of $\sJ_{n,1}, \ldots, \sJ_{n,d}$, the associated
  joint eigenvector of $z$ is $(\PP_0^\Tr(z), \ldots,
  \PP_{n-1}^\Tr (z))^\Tr$. Furthermore, if $z$ is a  
  complex-valued generalized joint eigenvalue, then $\bar z$ is also a
  generalized joint eigenvalue. 
\end{thm}

\begin{proof}
  The proof that a common zero is the same as a joint eigenvalue
  follows almost verbatim the proof for the positive definite case. In
  particular, $\PP_n$ can have at most $r_{n-1} =\dim\Pi_{n-1}^d$ zeros
  since $\sJ_{n,i}$ is of size $r_{n-1} \times r_{n-1}$. Now, since
  the moments are real, so are the matrices $\sJ_{n, i}$. If $z$ is a
  joint eigenvalue with joint eigenvector $\xi$, then $\sJ_{n,i} \xi = z
  \sS_n \xi$; taking the conjugate of this identity shows that $\bar
  z$ is also a joint eigenvalue with $\bar \xi$ as its joint
  eigenvector.
\end{proof}

In the positive definite case, $\sS_n$ is the identity matrix and the
generalized eigenvalues are just the ordinary
eigenvalues, so that the common zeros of $\PP_n$ are exactly the
eigenvalues of $\sJ_{n, i}$. Since $\sJ_{n, i}$ are symmetric, this
implies that all zeros are real when $\CL$ is positive definite. We
call $x \in \CC^d$ a simple zero of $\PP_n$ if the Jacobian $D \PP_n$ 
of the polynomials in $\PP_n : \RR^d \to \RR^{r_n^0}$ is nonsingular.

\begin{prop}
  If $x$ is a common zero of $\PP_n$ and $K_n(x,x) \ne 0$, then $x$
  must be simple and $x$ cannot be a zero 
  of $\PP_{n+1}$. 
\end{prop}
 
This is a consequence of the identity \eqref{eq:Knxx}. In the positive definite setting, $\sS_n$ is identity, 
so that $K_n(x,x) \ge K_0(x,x) =1$. Hence, all zeros of $\PP_n$ must be simple by \eqref{eq:Knxx} and, 
furthermore, $\PP_n$ and $\PP_{n+1}$ do not have common zeros. In the non-positive definite case, 
$K_n(x,x)$ may not be positive because the sign changes in $\sS_n$ and it can be zero. However, since 
$\sS_n$ is a signature matrix, it follows that
$$
   K_n(x,x) = \sum_{k=0}^n \sum_{|\a|  = k} \ve_\a |P_\a(x)|^2,  \qquad \ve_\a = \pm 1,
$$
so that $K_n(x,x)$ is real--valued even if $x \in \CC^d$. 

By Theorem \ref{thm:JointEigen}, the polynomials in $\PP_n$ can have at most $r_{n-1} = \dim \Pi_{n-1}^d$
common zeros. The following theorem characterizes when the number of common zeros is maximal. To 
prove it, we need some algebraic preparations, which we shall tailor to our needs. 

The \emph{ideal} generated by $\PP_n$ is defined as
$$
\left\langle \PP_n \right\rangle
:= \PP_n^\Tr \left( \Pi^d \right)^{r_n^0} = \left\{ \sum_{|\alpha|=n}
  P_\alpha Q_\alpha : Q_\alpha \in \Pi^d \right\}, \qquad \PP_n =
\left[ P_\alpha : |\alpha| = n \right],
$$
and $\PP_n$ is called an \emph{H-basis} provided that
\begin{equation}
  \label{eq:HBasis}
  \left\langle \PP_n \right\rangle \cap \Pi_m^d = \PP_n^\Tr \left(
    \Pi_{m-n}^d \right)^{r_n^0}, \qquad m \ge n,
\end{equation}
which means that any polynomial of degree $m$ in the ideal
$\left\langle \PP_n \right\rangle$ has a representation with respect
to $\PP_n$ in which the degree of no term in the sum exceeds
$m$. Since $\PP_n$
consists of $r_n^0$ linearly independent polynomials of degree $n$
whose leading forms span $\CP_n^d$, it follows that $\Pi^d /
\left\langle \PP_n \right\rangle = \Pi_{n-1}^d$ if and only if $\PP_n$ is
an H-basis. For more information on H-bases cf. \cite{Groebner,Sauer01}.
Moreover, $\PP_n$ defines a \emph{reduced multiplication} by the
factors $Q \mapsto (x_i Q)_{\PP_n}$ on $\Pi_{n-1}^d$ by setting
$$
(x_i Q)(x) = x_i Q(x) - \PP_n (x)^\Tr q,
$$
where the factor $q \in \RR^{r_n^0}$ is chosen such that $\deg \left(
  x_i Q(x) - \PP_n (x)^\Tr q \right) \le n-1$. In other words, if
$\deg x_i Q(x) = n$, then $q$ is chosen such that the homogeneous
leading form of $\PP_n (x)^\Tr$ coincides with the leading form of
$x_i Q(x)$, otherwise we have $q = 0$. Of course, $( x_i \cdot
)_{\PP_n}$ is a linear map from $\Pi_{n-1}^d$ to $\Pi_{n-1}^d$ and its
matrix representation with respect to the basis
$\PP_0,\dots,\PP_{n-1}$ is $\sM_{n,i} := \sS_n \sJ_{n,i}$,
$i=1,\dots,d$. More precisely, if $Q = \wh Q_0^T \PP_0 + \cdots + \wh
Q_{n-1}^T \PP_{n-1}$, then
$$
(x_i Q)_{\PP_n} = \left[ \wh Q_0^T, \dots, \wh Q_{n-1}^T \right] \sS_n \sJ_{n,i}
\left[
  \begin{array}{c}
    \PP_0 \\ \vdots \\ \PP_{n-1}
  \end{array}
\right].
$$
This is a
a direct consequence of multiplying both sides of \eqref{eq:3term} by
$S_n$ and the definition in \eqref{eq:truncJacobi}. The matrices
$\sM_{n,i}$ are called the \emph{multiplication tables} for the
coordinate multiplication modulo $\PP_n$. They can be used to
characterize whether $\PP_n$ is an H-basis; the following result is a special 
case of characterization for border bases, given for example in \cite{Mour}.

\begin{lem}\label{lem:HbasComm}
  $\PP_n$ is an H-basis if and only if the multiplication tables
  $\sM_{n,i}$ commute.
\end{lem}

Indeed, if $\PP_n$ is an H-basis, then $M_{n,i}$ describes the action
of multiplication by $x_i$ on $\Pi^d / \left\langle \PP_n
\right\rangle = \Pi_{n-1}^d$ and since multiplication is commutative,
so are the matrices that represent it. If, on the other hand, the
matrices commute, then for any $|\alpha| \le n-1$ the matrix
$\sM^\alpha = \sM_{n,1}^{\alpha_1} \cdots \sM_{n,d}^{\alpha_d}$ is
well defined and its first row contains the coefficients of the
monomial $x^\alpha$ with respect 
to $\PP_0,\dots,\PP_{n-1}$. Hence, there are $r_{n-1}$ linearly
independent reduced polynomials, hence $\dim \Pi^d / \left\langle \PP_n
\right\rangle \ge r_{n-1}$ and therefore $\PP_n$ is an H-basis and, in
particular, has $r_{n-1}$ common zeros counting multiplicities.

\begin{thm}\label{thm:commJn}
  The polynomials in $\PP_n$ have $r_{n-1}$ common zeros if and only if 
  \begin{align}\label{eq:commJn}
    A_{n-1,i} S_n A_{n-1,j}^\Tr = A_{n-1,j} S_n A_{n-1,i}^\Tr,
    \qquad 1 \le i, j \le d.
  \end{align}
\end{thm}

\begin{proof}
  According to Lemma~\ref{lem:HbasComm}, to have the maximal number of
  common zeros is equivalent to the commuting of the $\sS_n \sJ_{n,i}$ which
  is in turn equivalent to
  $\sJ_{n,i} \sS_n J_{n,j} = \sJ_{n,j} \sS_n J_{n,i}$, $1\le i \ne j
  \le d$. Using the commutativity of the Jacobi matrices, it is not
  difficult to see that we only need to consider the commutativity of 
  $Q_i S_n Q_j^\Tr$, where $Q_i$ and $Q_j$ denote the last row of
  $\sJ_{n,i}$ and $\sJ_{n,j}$, respectively; that is,  
  \begin{align*}
    A &_{n-2,i}^\Tr S_{n-2} A_{n-2,j} + B_{n-1,i} S_{n-1} B_{n-1,j} 
        = A_{n-2,j}^\Tr S_{n-2} A_{n-2,i} + B_{n-1,j} S_{n-1} B_{n-1,i}, 
  \end{align*}
  which is equivalent to \eqref{eq:commJn} by the third identity of
  \eqref{eq:comm} with $n$ replaced by $n-1$.
\end{proof}

The same argument can also be applied to an arbitrary recurrence
relation of the form \eqref{eq:MultiplArbit} with the multiplication
table
\begin{equation}
  \label{eq:truncJacobiGen}
  \sJ_{n,i}' :  =  \left[
    \begin{matrix} B_{0,i}& A_{0,i}&&&\bigcirc\cr
      C_{1,i} &B_{1,i}&A_{1,i}&&\cr  &\ddots&\ddots&\ddots&\cr
      &&C_{n-2,i} &B_{n-2,i}&A_{n-2,i}\cr
      \bigcirc&&&C_{n-1,i} &B_{n-1,i}  \end{matrix} \right],
  \qquad 1\le i\le d,
\end{equation}
which leads to the following variant of Theorem~\ref{thm:commJn}, that
in particular includes the monic orthogonal polynomials $\wh \PP_n$.

\begin{cor}\label{cor:commJnGen}
  The polynomials $\FF_n$ from \eqref{eq:MultiplArbit} have $r_{n-1}$
  common zeros if and only if
  \begin{equation}
    \label{eq:commJnGen}
    A_{n-2,i} \, C_{n-1,j} = A_{n-2,j} \, C_{n-1,i}, \qquad 1 \le i,j \le d.
  \end{equation}
\end{cor}

The conditions \eqref{eq:commJn} and \eqref{eq:commJnGen} are highly
non-trivial. For example, the condition does not hold if $\CL$ is centrally
symmetric, which means that $\CL$ satisfies 
$$
  \CL (x^\a) = 0, \qquad |\alpha| \in 2 \NN_0 + 1;
$$
if $\CL f = \int_\Omega f(x) W(x) \d x$, then $\CL$ is centrally symmetric if $W(x) = W(-x)$ whenever $x \in \Omega$ 
implies $-x \in \Omega$. It is known that if $\CL$ is positive definite and centrally symmetric, then $\PP_n$ has no 
common zeros for $n$ even and only one common zero (the origin) if $n$ is odd \cite[Section 3.7]{DX}. 

When $\PP_n$ has maximal common zeros, we can consider interpolation
and cubature rules as in one variable. 

\begin{prop} \label{prop:interp}
Assume \eqref{eq:commJn} that holds. Let $Z_n = \{(\xi_{k,1:}\ldots,
\xi_{k,d}),  1 \le k \le r_{n-1}\}$ be  
the set of common zeros of $\PP_n$. Assume that $\xi \in Z_n$ are not zeros of $\PP_{n+1}$. Then the 
unique polynomial, $L_n f$, in $\Pi_n^d$ that satisfies 
$$
   L_n f(\zeta) = f(\zeta), \qquad \zeta \in Z_n,
$$
for any generic function $f$ is given by the formula 
\begin{equation} \label{eq:Lagrange}
  L_n f(x) = \sum_{\zeta \in Z_n} f(\zeta)
  \frac{K_n(x,\zeta)}{K_n(\zeta,\zeta)}.
\end{equation}
\end{prop}

\begin{proof}
By our assumption on the zeros, $K_n(\zeta,\zeta) \ne 0$ for $\zeta
\in Z_n$. By Theorem \ref{thm:JointEigen}, if 
$\zeta$ is in $Z_n$, then so is $\bar \zeta$. If $\zeta,\eta\in Z_n$ and $\zeta \ne \eta$, then there must be a $j$ such
that $\zeta_j \ne \overline{\eta_j}$. Hence, by \eqref{eq:Knxy}, $K_n(\zeta,\eta) =0$. Consequently, it follows that
$L_n f$ in \eqref{eq:Lagrange} satisfies the interpolation conditions. 
\end{proof}

Applying $\CL$ to $L_n f$ leads to a cubature rule on $Z_n$, which is
a Gaussian cubature.  

\begin{thm} \label{thm:CF2}
Let $\CL (f) = \int_\Omega f(x) w(x) \d x$ for a weight $w$ defined on
$\Omega$ and assume that $Z_n$
is defined as in Proposition \ref{prop:interp}. Then 
\begin{equation} \label{eq:CF2}
  \int_\Omega f(x) w(x) \d x = \sum_{ \zeta \in Z_n}
  \frac{1}{K_n(\zeta,\zeta)} f(\zeta), \qquad f\in
  \Pi_{2n-1}^d. 
\end{equation}
Moreover, the cubature rule is real-valued for all polynomials.
\end{thm}

\begin{proof}
That the cubature rule is of degree $2n-1$ follows as in the positive definite case. If $\zeta \in Z_n$ is a complex
number, then $\bar \zeta \in Z_n$. Since $f$ is a real-valued polynomial, $\overline f(\zeta) = f(\bar \zeta)$. Moreover, 
$K_n(\zeta,\zeta) = K_n(\bar \zeta,\bar \zeta)$ is real, so that 
$$
 \frac{1}{K_n(\zeta,\zeta)} f(\zeta) +  \frac{1}{K_n\left(\bar \zeta,
     \bar \zeta \right)} \overline {f(\zeta)} =
 \frac{2}{K_n(\zeta,\zeta)} \Re f(\zeta) 
$$
is real-valued for each complex-valued $\zeta \in Z_n$. 
\end{proof}

Since the Gaussian cubature rule requires $\PP_n$ to have $r_{n-1}$ simple common zeros, they exist rarely and 
do not exist, for example, if $\CL$ is centrally symmetric. Nevertheless, they do exist in some cases. In the
next subsection, we recall a class of weight functions for which they exist. 

\subsection{Example of Gaussian cubature rules}\label{sect:3.4}
 We start with a weight function $w$ on $\RR$, which is definite and may change 
signs in its support set $[a,b]$. Let $p_n(w;x)$ be the orthogonal polynomial with respect to $w$ that satisfies
$$
  \int_{\RR} p_n(w;t) p_m(w;t) w(t) \d t = \ve_n \delta_{n,m}, \qquad n \ne m, \quad \ve = \pm 1.
$$
We assume that $p_n(w)$ has $n$ simple, real zeros in $[a,b]$ so that $w$ admits the Gaussian quadrature 
rule \eqref{eq:Gauss-quad} of degree $2n-1$, 
$$
  \int_\RR f(t)w(t) \d t = \sum_{k=1}^n \l_{k,n} f(t_{k,n}), \qquad \deg f \le 2n-1.
$$ 

We define two families of symmetric polynomials, indexed by 
partitions in 
$$
\Lambda^d = \{\l \in \NN_0^{d+1}: \l_1\ge \l_2 \ge \cdots \ge \l_d\}.
$$
Let $\mathcal{S}_d$ be the symmetric group of $d$ elements.
For $\a \in \Lambda^d$, define 
\begin{equation*}
      Q_{\alpha}^{-\f12}(x): = \sum_{\beta \in \mathcal{S}_d} p_{\alpha_1}(x_{\beta_1})\ldots p_{\alpha_{d}}(x_{\beta_{d}}).
\end{equation*} 
Then $\{Q_\a^{-\f12}: \a \in \Lambda^d\}$ is a family of symmetric polynomials and satisfies 
\begin{equation}\label{eq:-1/2OP-ortho}
\frac{1}{d!} \int_{\RR^d} Q_{\alpha}^{-\f12}(x)\,Q_{\beta}^{-\f12}(x)\, \prod_{i=1}^d w(x_i) \d x
   = m_1! \ldots m_{d'}! \,\ve_{\a_1}\ldots \ve_{\a_d} \delta_{\a,\b},
\end{equation}
where $d'$ is the number of distinct elements in $\a$ and $m_i$ is the number of occurrences of the $i$-th 
distinct element in $\alpha$. Moreover, define 
$$
     J_{\alpha}(x) := \det\, [p_{\alpha_i+d-i}(x_j)]_{i,j=1}^d \quad \hbox{and}\quad J(x) = \prod_{1 \le i < j \le d} (x_i-x_j).
$$
Then the family of polynomials $\{Q_\a^{\f12}: \a \in \Lambda^d\}$, defined by 
\begin{equation*}
      Q_{\alpha}^{\f12}(x) =  \frac{J_{\alpha}(x)}{J(x)},  \qquad \a \in \Lambda^d,
\end{equation*}
consists of symmetric orthogonal polynomials that satisfy 
\begin{equation}\label{eq:1/2OP-ortho}  
 \frac{1}{d!} \int_{\RR^d}  Q_{\alpha}^{\f12}(x)\,Q_{\beta}^{\f12}(x)\,
          \left[J(x)\right]^2 \prod_{i=1}^d w(x_i) \d x  = \ve_{\a_1}\ldots \ve_{\a_d} \delta_{\a,\b};
\end{equation}
see, for example, \cite[Sect. 5.4]{DX}. These symmetric polynomials can be mapped to a family of ordinary 
orthogonal polynomials. Let 
$$
     \mathcal{R} = \{x \in \RR^d \mid \, x_1< x_2 < \cdots <x_d, \, x_j \in \mathrm{supp}(w)\}. 
$$
The mapping $x \in \mathcal{R} \mapsto u \in \Omega$, defined by
\begin{equation}\label{eq:symm-map}
  u_k = e_k(x_1, \ldots x_d) := \sum_{1\le i_1 < \ldots < i_k \le d} x_{i_1} \cdots x_{i_d}, \quad k =1,2,\ldots, d, 
\end{equation} 
sends $\mathcal{R}$ onto a domain $\Omega$ and the Jacobian of the map is $J(x)$. Under the mapping 
\eqref{eq:symm-map}, the polynomials $Q_\a^{\pm \f 12}(x)$ become
polynomials in the variable(s) $u$, and these polynomials are
orthogonal with respect to the weight function $W_{\pm \f12}$ on
$\Omega$, where 
$$
  W_{-\f12}(u) = \prod_{i=1}^d w(x_i) [\Delta(u)]^{-\f12} \quad \hbox{and}\quad 
  W_{\f12}(u) = \prod_{i=1}^d w(x_i) [\Delta(u)]^{\f12}. 
$$
 
\begin{prop} 
Let $Q_\a^{\pm \f12}$ be the symmetric orthogonal polynomials in \eqref{eq:-1/2OP-ortho} and
\eqref{eq:1/2OP-ortho}. For $\a \in \Lambda^d$ and $n = \a_1 \ge \ldots \ge \a_d$, define 
$$
   P_\a^{n, -\f12} (u) =  Q_\a^{-\f12} (x) \quad \hbox{and} \quad  P_\a^{n, \f12} (u) =  Q_\a^{\f12} (x), 
$$
where $u = u(x)$ as in \eqref{eq:symm-map}. Then $\PP_n^{\pm \f12} = 
\{P_\a^{n, \pm \f 12}: n = \a_1 \ge \ldots \ge \a_d \ge 0\}$ is an orthogonal basis of $\CV_n(W_{\pm \f 12})$.
\end{prop}

The polynomials in $\{P_\a^{n,\pm \f12}: |\a| = n\}$ have $\dim \Pi_{n-1}^d$ common zeros and
they are nodes of the Gaussian cubature rules. More precisely, we have the following \cite{BSX}: 

\begin{thm} \label{thm:GaussEx1} 
The weight function $W_{\pm \f12}$ admits Gaussian cubature rules of degree $2n-1$ for $n =1,2,\ldots$. 
More precisely, 
\begin{equation} \label{eq:GaussEx1}
 \int_{\Omega} f(u) W_{-\f12}(u) \d u = \sum_{\g_1=1}^n \sum_{\g_2=1}^{\g_1} \cdots 
  \sum_{\g_d=1}^{\g_{d-1}} \Lambda_{\g,n}^{(-\f12)} f(u_{\g,n}), 
\end{equation}
where $u_{\g,n}$ is the image of $x_{\g,n} = (t_{\g_1,n}, t_{\g_2,n},\ldots, t_{\g_d,n})$ under the mapping 
\eqref{eq:symm-map}, and  
\begin{equation} \label{eq:GaussEx2}
 \int_{\Omega} f(u) W_{\f12}(u) \d u = \sum_{\g_1=1}^{n+d-1} \sum_{\g_2=1}^{\g_1-1} \cdots 
  \sum_{\g_d=1}^{\g_{d-1}-1} \Lambda_{\g,n}^{(\f12)} f(u_{\g,n+d-1}), 
\end{equation}
where $u_{\g,n+d-1}$ is the image of $x_{\g,n+d-1}=(t_{\g_1,n+d-1}, t_{\g_2,n+d-1},\ldots, t_{\g_d,n+d-1})$ 
under the mapping \eqref{eq:symm-map} and, furthermore, 
\begin{equation} \label{eq:coeffGaussEx}
   \Lambda_{\g,n}^{(-\f12)} = \frac{\l_{\hat \g_1,n}^{m_1} \dots \l_{\hat \g_{d'},n}^{m_{d'}}}{m_1!\cdots m_{d'}!}
   \quad \hbox{and} \quad  \Lambda_{\g,n}^{(\f12)} =J(t_{\g,n})^2 
         \l_{\g_1,n} \dots \l_{\g_{d},n},
\end{equation}
where $\hat \g_1 \ldots, \hat \g_{d'}$ denote the distinct elements in $\g$ and $m_1, \ldots, m_{d'}$ denote 
their respective multiplicities. 
\end{thm}

The proof of this theorem is purely algebraic and holds also for the case that $w$ may not be positive. 
For another example of the Gaussian cubature formula, see \cite{LX}, where the domain is the image of 
the simplex under the mapping of elementary symmetric functions. For $d =2$, the domain is bounded by
a hypocycloid. 

\section{Continued fraction expansions of moment sequences}
\setcounter{equation}{0}

To recover an analogy for the continued fractions in several variables, we let 
$\mu = \left( \mu_\alpha:  \alpha \in \NN_0 \right)$ be a moment
sequence such that $\det M_n \neq 0$, $n \ge 0$, and consider the
associated formal Laurent series 
$$
\mu (z) = \sum_{\a\in \NN_0^d} \mu_\alpha z^{-\alpha}, \qquad z \in z
\in (\CC \setminus 0)^d. 
$$  
If there is a rational function $r_n(z)$ of degrees $\le n$ for $n \ge
0$, such that
$$
\mu(z) - r_n (z) = O \left( z^{-2 n} \right), \qquad \text{i.e..}
\qquad \mu(z) - r_n (z) = \sum_{|\alpha| \ge 2n} \lambda_\alpha z^{-\alpha},
$$
then $r_n(z)$ can be regarded as an analogy for the continued fraction in several variables.

Our goals in this section are twofold. The first one is to construct a
sequence of moment  
sequences $\mu^n$, $n \ge 0$, iteratively from a polynomial sequence that satisfies the
three-term relations and show that $\mu^n \to \mu$, where $\mu$ is the
moment sequence 
that makes the polynomials orthogonal. The second goal is to define
and construct multivariate continued fractions based on $\mu^n$, and
to characterize the moment sequence for which these continued
fractions exist.
 
\subsection{Moments from three-term recurrence}
\label{sec:RecToMoments}

As shown in the previous section, if $\CL$ is a {\it definite} moment
functional and $\mu$ is its moments, then the orthogonal polynomials
satisfy the three-term relations \eqref{eq:MultiplArbit}, whose matrix
conditions in turn satisfy \eqref{eq:CommMultGen},
\eqref{eq:commJnGen}, and \eqref{eq:rank-cond}.

We start from the three-term recurrence relation satisfied by some polynomial sequence. Let 
$\FF_n \in ( \Pi_n^d )^{r_n^0}$, $n \ge 0$, be a sequence of polynomials that satisfies 
\begin{equation}
  \label{eq:CFThreeTerm}
  x_i \FF_n (x) = A_{n,i} \FF_{n+1} (x) + B_{n,i} \FF_n (x) + C_{n,i}
  \FF_{n-1} (x), \qquad i = 1,\dots,d,
\end{equation}
and whose coefficient matrices satisfy  \eqref{eq:CommMultGen}, \eqref{eq:commJnGen}
and the rank condition \eqref{eq:rank-cond}. 
By Theorem \ref{thm:quasidetFun}, there exists a quasi-determinate linear
functional $\CL$ such that the $\FF_n$ form an orthogonal basis with respect to $\CL$, i.e.,
\begin{equation}
  \label{eq:FFnOrthogonal}
  \CL ( Q \FF_n ) = 0, \quad Q \in \Pi_{n-1}^d, \qquad \det \left[\CL \left( \FF_n\FF_n^\Tr\right) \right]\neq 0,
\end{equation}
holds for any $n \ge 0$. We normalize $\CL$ such that $\CL(1) = 1$ and
denote by $\mu_\alpha = \CL \left((\cdot)^\alpha \right)$, $\alpha \in
\NN_0^d$, the moments of this functional. In this setting, the moment sequence 
$\mu$ is well-defined but not explicitly given. 
 
We construct a sequence $\mu^n$, $n \in \NN_0$, of moment sequences 
recursively in $n$ and based entirely on the recurrence
\eqref{eq:CFThreeTerm}. This will be done in such a way that
$\mu^n$ converges to $\mu$. In other words, we provide a constructive method
for the moment sequence $\mu$ from $\FF_n$ defined by
\eqref{eq:CFThreeTerm}. 

By Corollary~\ref{cor:commJnGen}, all $\FF_n$, 
$n \ge 0$, are H-bases and $\Pi^d / \left\langle \FF_n \right\rangle = \Pi_{n-1}^d$ is the 
associated interpolation space modulo $\left\langle \FF_n \right\rangle$. This implies 
that $\FF_n$ has $r_{n-1}$ common zeros, counting multiplicities. Recall that the multiplicity
of common zeros of an ideal is not a simple number anymore, but a structural quantity, 
namely a finite dimensional subspace of $\Pi$ which is closed under
differentiation or \emph{$D$-invariant} for short. More precisely,
$\CQ$ is $D$-invariant if $Q \in \CQ$ implies that
$\frac{\partial}{\partial x_j} Q \in \CQ$ as well; in particular, this implies that $1 \in \CQ$. 
Therefore, that $\FF_n$ has $r_{n-1}$ common zeros is equivalent to that there exist a 
set $Z_n$ of finitely many points in $\CC^d$ and $D$-invariant subspaces $\CQ_{n,\zeta}$, 
$\zeta \in Z_n$, describing the multiplicity, such that $P \in \left\langle \FF_n \right\rangle$ 
if and only if
\begin{equation} \label{eq:DualIdealCond}
  \left( Q(D) P \right) (\zeta) = 0, \qquad Q \in \CQ_{n,\zeta}, \quad
  \zeta \in Z_n,
\end{equation}
and $\sum_{\zeta \in Z_n} \dim \CQ_\zeta = r_{n-1}$
(cf. \cite{groebner37:_ueber_macaul_system_bedeut_theor_differ_koeff}).
Let $Q_{n,\zeta}$ be a graded basis of $\CQ_{n,\zeta}$ that thus contains the
constant polynomial. Then the interpolation problem with respect to
$\left( Q(D) P \right) (\zeta)$, $Q \in \CQ_{n,\zeta}$, $\zeta \in Z_n$, 
is \emph{poised}, i.e., uniquely solvable, in $\Pi_{n-1}$. Equivalently, 
the \emph{Vandermonde matrix} 
\begin{equation}
  \label{eq:VandermondeQDef}
  V_n := \left[ \left( Q(D) (\cdot)^\alpha \right) (\zeta) :
    \begin{array}{c}
      Q \in \CQ_{n,\zeta}, \, \zeta \in Z_n \\
      |\alpha| \le n-1
    \end{array}
  \right]  
\end{equation}
is nonsingular. Note that the case of simple zeros is described in
this way as
$\dim \CQ_{n,\zeta} = 1$ and $\CQ_{n,\zeta} = \CC$, and the above
Vandermonde matrix is the classical one for the \emph{Lagrange
  interpolation} problem, in general \eqref{eq:VandermondeQDef}
corresponds to a \emph{Hermite interpolation problem}, however.

To reconstruct the moment sequence, the \emph{$\Theta$-operator} is
useful. It is defined as
$$
\theta_j := (\cdot)_j \frac{\partial}{\partial x_j}, \quad
j=1,\dots,d, \qquad
\Theta^\alpha := \theta_1^{\alpha_1} \cdots \theta_d^{\alpha_d},
$$
and $Q(\Theta)$, $Q \in \Pi^d$, is defined accordingly. It has been shown, for
example in \cite{Sauer2017:_Reconstruction}, that $V_n$ is nonsingular
if and only if the \emph{$\Theta$-Vandermonde matrix}
$$
V_{\Theta,n} := \left[ \left( Q(\Theta) (\cdot)^\alpha \right) (\zeta) :
  \begin{array}{c}
    Q \in \CQ_{n,\zeta}, \, \zeta \in Z_n \\
    |\alpha| \le n-1
  \end{array}
\right]
$$
is nonsingular. Moreover, by construction,
$$
\left[ \left( Q(\Theta) (\cdot)^\alpha \right) (\zeta) :
  \begin{array}{c}
    Q \in \CQ_{n,\zeta}, \, \zeta \in Z_n \\
    |\alpha| \le n
  \end{array}
\right] \wh \Fb_n = 0
$$
since $\FF_n$ is an H-basis for the ideal defined by the dual conditions
in \eqref{eq:DualIdealCond}.

We now use the \emph{Prony's method} to define a sequence of moment 
sequences $\mu^n$, $n \ge 0$, by iteration, which is defined by using the recurrence 
\eqref{eq:CFThreeTerm}. 
To that end, we note that $\FF_0 = 1$ and set, accordingly $\mu^0_\alpha
= 1$, $\alpha \in \NN_0^d$. Next, suppose that we have already computed a moment
sequence $\mu^{n-1}$.  To define $\mu^n$, we first find an exponential
polynomial $g_n$ of the form
\begin{equation}
  \label{eq:ExpoPoly}
  g_n (x) := \sum_{\zeta \in Z_n} G_{n,\zeta} (x) \, \zeta^x, \qquad G_{n,\zeta}
  \in \CQ_{n,\zeta},
\end{equation}
such that
\begin{equation} \label{eq:fnmunEq}
  g_n (\alpha) = \mu^{n-1}_\alpha, \qquad |\alpha| \le n-1.  
\end{equation}
Expanding $G_{n,\zeta}$ with respect to the basis $Q_{n,\zeta}$ of $\CQ_{n,\zeta}$, we
note that
$$
g_n (\alpha) = \sum_{\zeta \in Z_n} \Big(\sum_{Q \in \CQ_{n,\zeta}} c_Q Q
(\alpha) \Big) \, \zeta^\alpha
= \sum_{\zeta \in Z_n} \sum_{Q \in \CQ_{n,\zeta}} c_Q \left( Q
(\Theta) (\cdot)^\alpha \right) (\zeta),
$$
hence the coefficients $c_Q$, when \eqref{eq:fnmunEq} is satisfied, are obtained by solving 
the linear system
$$
\left[ \mu^{n-1}_\alpha : |\alpha| \le n-1 \right] =
V_{\Theta,n}^T \left[ c_Q : Q \in Q_{n,\zeta}, \zeta \in Z_n \right].
$$
Since $V_{\Theta,n}$ is square and nonsingular, the coefficients $c_Q$, therefore $g_n$, are 
uniquely defined. Thus, once $\mu^{n-1}$ is present, the exponential
polynomial $g_n$ that satisfies \eqref{eq:fnmunEq} is uniquely
determined by $\mu^{n-1}$.

We are now ready to define the moment sequence $\mu^n$, which is simply
$$
\mu^n_\alpha = g_n (\alpha), \qquad \alpha \in \NN_0^d.
$$
The polynomials $\FF_n$ generate the so-called \emph{Prony ideal} for
the sequence $\mu^n$ and the associated Hankel matrices
$$
M^n_k := \left[ \mu^n_{\alpha+\beta} :
  \begin{array}{c}
    |\alpha| \le k \\ |\beta| \le k \\
  \end{array}
\right]
$$
can be factored into
$$
M^n_k = V_{\Theta,n,k}^T T_n V_{\Theta,n,k}, \qquad
V_{\Theta,n,k} := \left[ \left( Q(\Theta) (\cdot)^\alpha \right) (\zeta) :
  \begin{array}{c}
    Q \in Q_{n,\zeta}, \, \zeta \in Z_n \\
    |\alpha| \le k
  \end{array}
\right],
$$
where $T_n \in \RR^{r_{n-1} \times r_{n-1}}$ is a nonsingular block
diagonal matrix with blocks of the size $\dim \CQ_\zeta$, $\zeta \in
Z_n$, see \cite{Sauer2017:_Reconstruction}.
In particular, $\det M^n_{n-1} \neq 0$, i.e., $\mu^n$ is
definite up to degree $n-1$ and $M^n_n \wh \Fb_n
= 0$. In the same way, the infinite Hankel matrix $M^n$ factorizes into
$$
M^n = V_{\Theta,n,\infty}^T T_n V_{\Theta,n,\infty}, \qquad
V_{\Theta,n,\infty} := \left[ \left( Q(\Theta) (\cdot)^\alpha \right) (\zeta) :
  \begin{array}{c}
    Q \in Q_{n,\zeta}, \, \zeta \in Z_n \\
    |\alpha| \in \NN_0^d
  \end{array}
\right],
$$
hence the associated Hankel operator has rank $r_{n-1}$.

\begin{thm}\label{thm:MomentCoinc}
  For $n \ge 1$, the moments $\mu^n$ and $\mu$ from above satisfy
  \begin{equation}
    \label{eq:MomentCoinc}
    \mu_\alpha = \mu^n_\alpha, \qquad |\alpha| \le 2n-1.
  \end{equation}
\end{thm}

\begin{proof}
  For $n \ge 1$ we have that
    $\mu^n_\alpha = \mu^{n-1}_{\alpha} = \mu_\alpha$
  for $|\alpha| \le n-1$; for $n = 1$ this is the fact that $\mu^0_0 =
  \mu_0 = 1$ and for $n \ge 2$ it immediately follows by induction.
  We now set $\omega_\alpha
  := \mu_\alpha - \mu^\alpha$, $\alpha \in \NN_0^d$, hence
  $\omega_\alpha = 0$, $|\alpha| \le n-1$. We partition the
  coefficients of $\FF_n$ as
  $$
  \wh \Fb_n = \left[
    \begin{array}{c}
      \wh F_{n,n-1} \\ \wh F_{n,n}
    \end{array}
  \right], \qquad \wh F_{n,n-1} \in \RR^{r_n^0 \times r_{n-1}}, \,
  F_{n,n} \in \RR^{r_n^0 \times r_n^0},
  $$
  with $\det \Fb_{n,n} \neq 0$ and note that due to the orthogonality
  of $\FF_n$ we have $M_n \wh \Fb_n = 0$, hence
  $$
  0 = \left( M_n - M_n^n \right) \wh \Fb_n = W_n \wh \Fb_n = \left[
    \begin{array}{cc}
      W_{n-1}^0 & W_n^0 \\
      \vdots & \vdots \\
      W_{n-1}^n & W_n^n
    \end{array}
  \right] \left[
    \begin{array}{c}
      \wh \Fb_{n,n-1} \\ \wh \Fb_{n,n}
    \end{array}
  \right],
  $$
  where
  $$
  W_{n-1}^k = \left[ \omega^n_{\alpha+\beta} :
    \begin{array}{c}
      |\alpha| = k \\ |\beta| \le n-1 \\
    \end{array}
  \right], \quad W_{n}^k = \left[ \omega^n_{\alpha+\beta} :
    \begin{array}{c}
      |\alpha| = k \\ |\beta| = n \\
    \end{array}
  \right],
  $$
  which yields that
  $$
  W_n^k = - W_{n-1}^k \wh \Fb_{n,n-1} \wh \Fb_{n,n}^{-1}.
  $$
  Then $W_{n-1}^0 = 0$ implies that $W_n^1 = 0$ and since $W^{k+1}_{n-1} =
  \left[ W^k_{n-1} \, W_n^k \right]$, a simple induction on $k$ yields
  that
  $$
  0 = W_{n-1}^n = \left[ \omega^n_{\alpha+\beta} :
    \begin{array}{c}
      |\alpha| = n \\ |\beta| \le n-1 \\
    \end{array}
  \right].
  $$
  Since any $\alpha$ with $|\alpha| \le 2n-1$ can be written as
  $\beta + \gamma$, $|\beta| \le n$, $|\gamma| \le n-1$, the claim follows.
\end{proof}

At this point, it is worthwhile to remark that the above construction
inductively computes the initial segments $\mu^n$ of $\mu$ without
knowing $\mu$ at any point of this process. This holds in accordance
with the nature of 
continued fractions.

In the case that the Gaussian cubature rules exist, the moment sequence $\mu^n$ that satisfies
\eqref{eq:MomentCoinc} has a simple formulation. Indeed, the existence of the Gaussian cubature 
rule means that all zeros in $Z_n$ are simple, so that $\dim \CQ_{n,\zeta} =1$ and, consequently, 
$G_{n,\zeta}$ are scalars. Moreover, they are exactly equal to the coefficients in the Gaussian cubature 
rules, which are $K_n(\zeta,\zeta)^{-1}$, with $K_n$ being the reproducing kernel defined in
\eqref{eq:reprod-kernel}. Indeed, let $g_n(x)= \sum_{\zeta \in Z_n} \frac{1}{K_n(\zeta,\zeta)} \zeta^x$. 
Then, the Gaussian cubature rule \eqref{eq:CF2} implies immediately that \eqref{eq:fnmunEq} holds for
every positive integer $n$. We state the result as a corollary. 
\begin{cor}\label{cor:4.2}
Let $\CL(f) = \int_\Omega f(x) w(x) \d x$ for a weight function defined on $\Omega \subset \RR^d$, for 
which the Gaussian cubature rule \eqref{eq:CF2} holds. Then the moment $\mu^n$ in Theorem 
\ref{thm:MomentCoinc} is precisely given by 
$$
  \mu_\a^n = \sum_{\zeta \in Z_n} \frac{1}{K_n(\zeta,\zeta)} \zeta^\a, \qquad \a \in \NN_0^d. 
$$
\end{cor}
We need the Theorem \ref{thm:MomentCoinc} in its full generality for the characterization 
of rational approximation for the Laurent series in the last subsection.

\subsection{Rational approximation to Laurent series}
\label{sec:RatFun}

Continued fractions of univariate polynomials have rational functions
as convergents. The same holds for our three-term expansions 
here, provided the recurrence matrices satisfy \eqref{eq:comm} and
\eqref{eq:commJn} or \eqref{eq:CommMultGen} and \eqref{eq:commJnGen},
respectively. To that end, we use the exponential polynomial $g_n$
\eqref{eq:ExpoPoly} from which $\mu^n$ is generated as $\mu^n_\alpha =
g_n (\alpha)$, and consider the formal Laurent series
\begin{equation} \label{eq:LaurenSeries}
\mu^n (z) := \sum_{\alpha \in \NN_0^d} \mu_\alpha^n \, z^{-\alpha}, \qquad
z \in (\CC \setminus 0)^d.
\end{equation}
It has already been shown by Power \cite{Power} that the Laurent series
associated with any finite rank Hankel operator induces a rational
Laurent series. In addition, one can even derive an explicit
representation of this rational function from \eqref{eq:ExpoPoly}. Recall
that $G_{n,\zeta}$ are coefficients of $g_n(x)$ and $\Theta$ is the theta
operator. 

\begin{prop}\label{prop:RatFunExplicit}
For the sequences $\mu^n$ given as $\mu^n_\alpha = g_n (\alpha)$, we have
  \begin{equation}\label{eq:muRatForm}
    \mu^n (z) = \sum_{\zeta \in Z_n} \left( G_{n,\zeta} (-\Theta)
      \frac{(\cdot)^\epsilon}{\left( (\cdot) - \zeta \right)^\epsilon}
    \right) (z), \qquad 
    \epsilon := (1,\dots,1).
  \end{equation}
\end{prop}

\begin{proof}
  By \eqref{eq:ExpoPoly}, we have that
  \begin{align*}
    \mu^n (z)
    &= \sum_{\alpha \in \NN_0^d} \sum_{\zeta \in Z_n} G_{n,\zeta}
      (\alpha) z^{-\alpha} \, \zeta^\alpha
      = \sum_{\alpha \in \NN_0^d} \sum_{\zeta \in Z_n} \left( G_{n,\zeta} (-\Theta)
      \left( \zeta \cdot (\cdot)^{-\epsilon} \right)^{\alpha} \right)
      (z) \\
    &= \sum_{\zeta \in Z_n} \left( G_{n,\zeta} (-\Theta)
      \sum_{\alpha \in \NN_0^d} \left( \zeta \cdot (\cdot)^{-\epsilon}
      \right)^{\alpha} \right) (z)
      = \sum_{\zeta \in Z_n} \left( G_{n,\zeta} (-\Theta) \frac{1}{1 - \zeta
      \cdot (\cdot)^{-\epsilon}} \right) (z) \\
    &= \sum_{\zeta \in Z_n} \left( G_{n,\zeta} (-\Theta)
      \frac{(\cdot)^\epsilon}{\left( (\cdot)
      - \zeta \right)^\epsilon} \right) (z),
  \end{align*}
  which is \eqref{eq:muRatForm}.
\end{proof}

In the case of simple zeros, or the existence of Gaussian cubature rules, we deduce from Corollary \ref{cor:4.2}
the following: 
\begin{cor} \label{cor:rational_mu}
Under the assumption of Corollary \ref{cor:4.2}, the rational function $\mu_n(z)$ in \eqref{eq:muRatForm}
is given by 
\begin{equation}\label{eq:rational_mu}
\mu^n (z) = \sum_{\zeta \in Z_n} \frac{1}{K_n(\zeta,\zeta)} \frac{z^\epsilon}{(z- \zeta)^\epsilon}
= \sum_{\zeta \in Z_n}  \frac{1}{K_n(\zeta,\zeta)} \frac{z_1 \cdots z_d}{(z_1-\zeta_1) \cdots (z_d - \zeta_d)}.
\end{equation}
\end{cor} 

This shows that the rational functions coming from moment sequences
or finite rank Hankel operators have some special properties: they
have finitely many poles of order $d$ and their denominator
polynomials can be factored into simple linear factors. And obviously
the denominators are related to the ideal $\left\langle \FF_n
\right\rangle$ by the fact that the denominator is the product of the
common zeros of this ideal.

Let us consider the Laurent series and its rational approximants for
the examples in Subsection \ref{sect:3.4}
which admit the Gaussian cubature rules. Instead of dealing with integrals over the domain $\Omega$ there,
we work with integrals over $\RR^d$ using the mapping \eqref{eq:symm-map}.
 
Let $w$ be a weight function defined on $[-1,1]$ and assume that it
admits a Gaussian quadrature formula. Let
$$
  \CL_{-\f12}(f) = \frac1{d!} \int_{\RR^d} f(e_1(x),\ldots, e_d(x)) \prod_{i=1}^d w_i(x) \d x
$$
and 
$$
  \CL_{\f12} (f) = \frac1{d!} \int_{\RR^d} f(e_1(x),\ldots, e_d(x)) \prod_{1\le i< j \le d} (x_i-x_j)^2 
     \prod_{i=1}^d w_i(x) \d x,
$$
where $e_k(x_1,\ldots,x_d)$ is the $k$-th elementary symmetric function defined in \eqref{eq:symm-map}. 
For the moment sequence $\{ \CL_{\pm \f12}(x^\a) \}$, the Laurent series \eqref{eq:LaurenSeries} can be 
easily seen to satisfy 
$$
    \mu_{-\f12}(z) = \frac{1}{d!} \int_{\RR^d} \frac{z_1\cdots z_d}{(z_1 - e_1(x)) \cdots (z_d - e_d(x))} \prod_{i=1}^d w_i(x) \d x
$$
and 
$$
    \mu_{\f12}(z) = \frac{1}{d!} \int_{\RR^d} \frac{z_1\cdots z_d}{(z_1- e_1(x)) \cdots (z_d-e_d(x))} 
     \prod_{1\le i< j \le d} (x_i-x_j)^2  \prod_{i=1}^d w_i(x) \d x.
$$
By \eqref{eq:GaussEx1} and \eqref{eq:GaussEx2}, the corresponding rational functions that approximate 
$\mu_{\pm \f12}(z)$, as in \eqref{eq:rational_mu}, are given by  
\begin{equation} \label{eq:rationaEx1}
  \mu_{-\f12}^n(z) = \sum_{\g_1=1}^n \sum_{\g_2=1}^{\g_1} \cdots 
  \sum_{\g_d=1}^{\g_{d-1}} \Lambda_{\g,n}^{(-\f12)} \frac{z_1\cdots z_d} 
     {(z_1 - u_{\g_1,n})\cdots (z_d - u_{\g_d,n}) }
\end{equation}
and 
\begin{equation} \label{eq:rationalEx2}
  \mu_{\f12}^n(z) =  \sum_{\g_1=1}^{n+d-1} \sum_{\g_2=1}^{\g_1-1} \cdots 
  \sum_{\g_d=1}^{\g_{d-1}-1} \Lambda_{\g,n}^{(\f12)}  \frac{z_1\cdots z_d} 
     {(z_1 - u_{\g_1,n+d-1})\cdots (z_d - u_{\g_d,n+d-1})},
\end{equation}
where $u_{\g,n}$ and $u_{\g,n+d-1}$ are nodes of the corresponding Gaussian cubature rules and 
$\Lambda_{\g,n}^{(\pm \f12)}$ are given by \eqref{eq:coeffGaussEx}. Evidently, $\mu_{\pm \f12}^n(z)$
are rational functions and, by Theorem \ref{thm:MomentCoinc}, they satisfy 
$$
  \mu_{\pm \f12}(z) - \mu_{\pm \f12}^n(z) = O(|z|^{2n}).
$$

\subsection{Continued fraction expansions of moment sequences}
\label{sec:MomentCFExp}

We can now assemble the theory of the preceding sections to give a
characterization of moment sequences that admit an associated
multivariate continued fraction expansion.

To that end, we start with a sequence $\mu = \left( \mu_\alpha :
  \alpha \in \NN_0 \right)$ such that $\det M_n \neq 0$, $n \ge 0$. Define
$$
\mu (z) = \sum_{\a\in \NN_0^d} \mu_\alpha z^{-\alpha}, \qquad z \in z \in (\CC \setminus 0)^d.
$$  
As shown in Proposition~\ref{prop:SignOrthognal}, there exists a
sign-orthonormal basis $\PP_n$, $n \ge 0$, of $\Pi^d$ that satisfies
the recurrence relation \eqref{eq:3term} and the rank condition
\eqref{eq:rank-cond}. Moreover, the commutativity conditions
\eqref{eq:comm} hold for the recurrence matrices. If, in addition, the
commuting conditions \eqref{eq:commJn} are satisfied for $n \ge 0$,
then the polynomials $\PP_n$ are H-bases and there exist sequences
$\mu^n = \left( \mu^n_\alpha : \alpha \in \NN_0 \right)$ such that
$\mu^n_\alpha = \mu_\alpha$, $|\alpha| \le 2n-1$, and therefore
$$
\mu (z) - \mu^n (z) = O \left( z^\alpha : |\alpha| = 2n \right),
$$
where $\mu^n (z)$ is a rational function by \eqref{eq:muRatForm}. This
is the perfect analogy for the continued fraction expansion of a given
Laurent series.

The above procedure can be performed for \emph{any} vector polynomial
sequence $\FF_n \in \Pi_n^{r_n^0}$  given by a three-term recurrence
$$
\FF_{n+1} = A_n (x) \FF_n + B_n \FF_n + C_n \FF_{n-1}, \qquad A_n
(x) = \sum_{i=1}^d A_{n,i} \, x_i
$$
with matrix coefficients as in the recurrence relations of orthogonal polynomials.
Moreover, we let $\CA_n \subset \NN_0^d$ be a lower set, i.e., $\alpha \in \CA$
implies $\beta$ in $\CA$ whenever $\beta \le \alpha$ componentwise that has 
the property that $\{ x^\alpha : \alpha \in \CA_n \}$ spans the quotient space 
$\Pi^d / \left\langle \FF_{n-1} \right\rangle$. With $Z_n$ denoting the set of 
common zeros of $\FF_{n+1}$, we then define $g_n$ of the form \eqref{eq:ExpoPoly} 
by requiring that
\begin{equation}
  \label{eq:GenContFracmuDef}
  g_n (\alpha) = \mu^{n-1}_\alpha, \qquad \alpha \in \CA_n.
\end{equation}
This interpolation problem is uniquely solvable (cf. \cite{Sauer2017:_Reconstruction})
and again yields $\mu^n$ by sampling $g_n$: $\mu^n_\alpha = g_n (\alpha)$, $\alpha \in
\NN_0^d$. Finally, obtain by this procedure
a sequence of rational functions $\mu^n (z)$ whose explicit expression is given
in Proposition~\ref{prop:RatFunExplicit}. Now we can give a formal
definition of continued fractions that coincides with the one from
\cite{Perron,S_CF} for $d=1$. 

\begin{defn}
  \label{D:ContFracDef}
  For matrices $A_{n,i} : r_{n+1}^0 \times r_n^0$, $i=1,\dots,d$, $B_n
  : r_{n+1}^0 \times r_n^0$, $C _n : r_{n+1}^0 \times r_{n-1}$, $n \in
  \NN$, we define the \emph{continued fraction}
  \begin{equation}
    \label{eq:ContFracDef}
    \frac{ \left. C_1 \right|}{\left. A_1 (z) + B_1 \right|} + \cdots
    + \frac{ \left. C_n \right|}{\left. A_n (z) + B_n \right|} :=
    \sum_{\zeta \in Z_n} \left( G_{n,\zeta} (-\Theta)
      \frac{(\cdot)^\epsilon}{\left( (\cdot) - \zeta \right)^\epsilon}
    \right) (z),
  \end{equation}
  where $Z_n$ are the common zeros of $\FF_{n+1}$ and the
  $G_{n,\zeta}$ are determined iteratively by \eqref{eq:GenContFracmuDef}. The
  ideal $\left\langle \FF_{n+1} \right\rangle$ is called the
  \emph{denominator ideal} of the continued fraction.
\end{defn}

The left-hand side of \eqref{eq:ContFracDef} is a symbolic notation for 
the continued fraction in the spirit of the classical notation of one variable. 
In words, with each matrix sequence, the continued fraction expansion 
associates a sequence of rational functions, which is precisely the continued 
fraction expansion of rational functions, exactly like the case for $d=1$. 
Note, however, that the situation is significantly more intricate in several 
variables since $\FF_{n+1}$ may \emph{not} be an H-basis of the 
denominator ideal, and even $Z_n = \emptyset$ is possible. Hence, in 
contrast to the univariate case, we cannot expect these continued 
fractions to exist for all moment functionals. Finally, we find it worthwhile 
to mention that $\left\langle \FF_{n+1} \right\rangle$ is the so-called 
\emph{Prony ideal} for the associated exponential polynomial $g_n$.

We are now ready to give a fundamental definition.

\begin{defn}\label{def:ContinuedF}
 A Laurent series $\mu(z)$ has an \emph{associated continued fraction} 
 with coefficients $A_{n,i}, B_{n}, C_{n}$, $n \in \NN$ if
  $$
  \mu (z) - \left( \frac{ \left. C_1 \right|}{\left. A_1 (z) + B_1
      \right|} + \cdots + \frac{ \left. C_n \right|}{\left. A_n (z) +
        B_n \right|} \right) = O \left(z^\alpha : |\alpha| = 2n \right),
  \qquad n \in \NN.
  $$
\end{defn}

In our construction of signed orthogonal polynomials above, the
Laurent ideals, hence $Z_n$ and 
$\CQ_\zeta$, $\zeta \in Z_n$, depend only on the \emph{ideal} $\left\langle \PP_n
\right\rangle$, hence we can also start with a sequence of orthogonal
polynomials defined by the implicit recursion \eqref{eq:3term}, and
then refer to the explicit formula \eqref{eq:recursiveP}, use the notation
$$
G_n (z) = \sum_{i=1}^d G_{n,i} z_i = \sum_{i=1}^d D_{n,i}^\Tr S_n z_i, 
$$
and consider the continued fractions
\begin{equation}
  \label{eq:ContFraDef}
  \frac{-F_1 |}{| G_1(z) - E_1} + \frac{-F_2 |}{| G_2 (z) - E_2}
  + \cdots \frac{-F_n |}{| G_n (z) - E_2}, \qquad n \in \NN.
\end{equation}
For the continued fraction expansion defined in Definition \ref{def:ContinuedF},
we can now summarize our findings in the following way.

\begin{thm}
A Laurent series $\mu (z) = \sum \mu_\alpha z^{-\alpha}$ has an
  associated continued fraction expansion if and only if
  \begin{enumerate}[(i)]
  \item the Hankel matrices $M_n$ are definite, i.e., $\det M_n \neq
    0$, $n \ge 0$, and
  \item the coefficient matrices of the associated sign-orthonormal
    polynomials satisfy \eqref{eq:comm}, \eqref{eq:rank-cond} and
    \eqref{eq:commJn}. 
  \end{enumerate}
\end{thm}

As an example, the rational functions $\mu_{\pm \f12}^n$ given in 
\eqref{eq:rationaEx1} and \eqref{eq:rationalEx2} are the associated 
continued fraction expansions of the Laurent series $\mu_{\pm \f12}(z)$. 

Once more we want to emphasize that the critical condition is 
\eqref{eq:commJn}. Whenever $\det M_n \neq 0$, then the rank condition
\eqref{eq:rank-cond} and the commuting conditions \eqref{eq:comm}
follow automatically, they are needed, however, to ensure that a three
term recurrence yields a well-defined multiplication.

To further understand the meaning of \eqref{eq:commJn}, we consider
the \emph{monic} orthogonal basis $\wh \PP_n$ associated to the
definite sequence $\mu$, where $\wh \Pb_n = \left[
  \begin{array}{c}
    * \\ I_{r_n^0 \times r_n^0}
  \end{array}
\right]$. Here, the matrices $A_{n,i}$, $B_{n,i}$ and $C_{n,i}$ of the
recurrence relation \eqref{eq:MultiplArbit} can be given explicitly
from the moment matrix, cf. \cite{BarrioPenaSauer10}, in particular
$$
A_{n,i} = L_{n,i} =: \sum_{|\alpha| = n} e_\alpha
e_{\alpha+\epsilon_i}^T,
\qquad \text{and} \qquad
C_{n,i} = H_n L_{n-1,i}^T H_{n-1}^{-1},
$$
where
$$
H_n = M_{n,n} - M_{n,n-1}^\Tr M_{n-1}^{-1} M_{n,n-1}, \qquad
M_n = \left[
  \begin{array}{cc}
    M_{n-1} & M_{n,n-1} \\ M_{n,n-1} & M_{n,n}
  \end{array}
\right],
$$
denotes the Schur complement of $M_{n-1}$ in $M_n$, which is
nonsingular since $M_n$ and $M_{n-1}$ are nonsingular. For these
matrices, \eqref{eq:commJnGen} (with $n$ replaced by $n+1$ for
convenience) becomes
$$
L_{n-1,i} H_n L_{n-1,j}^T H_{n-1}^{-1} = L_{n-1,j} H_n L_{n-1,i}^T
H_{n-1}^{-1},
$$
or, equivalently,
\begin{equation*}
  L_{n-1,i} H_n L_{n-1,j}^T = L_{n-1,j} H_n L_{n-1,i}^T.
\end{equation*}
By \cite[Lemma 4.2]{X94}, this identity implies that $H_n$ is a Hankel matrix. that is, 
$(H_n)_{\alpha,\beta} = h_{\alpha+\beta}$. 
Since $M_{n,n}$ is a Hankel matrix, we obtain the following characterization: 

\begin{thm}
A sequence $\mu$ has an associated continued fraction
expansion if and only if it is definite and the matrix 
$$
 M_{n,n-1}^\Tr M_{n-1}^{-1} M_{n,n-1}
$$
is a Hankel matrix for all $n \ge 1$.
\end{thm}

\end{document}